\documentclass[12pt]{amsart}
\usepackage{amsmath,amssymb,amsthm,upref,graphicx,mathrsfs}

\usepackage{color}
\usepackage[
  colorlinks=true,
  linkcolor=blue,
  citecolor=blue,
  urlcolor=blue]{hyperref}

\numberwithin{equation}{section}

\newtheorem{theorem}{Theorem}[section]

\newtheorem{lemma}[theorem]{Lemma}

\theoremstyle{definition}
\newtheorem{definition}[theorem]{Definition}

\newtheorem{remark}[theorem]{Remark}

\newcommand{\be}{\begin{eqnarray*}}
\newcommand{\ee}{\end{eqnarray*}}
\newcommand{\beq}{\begin{equation}}
\newcommand{\eeq}{\end{equation}}

\begin{document}

\title[\bf Martingale Inequalities in Variable Exponent
Hardy spaces]
  {\bf Martingale Inequalities in Variable Exponent
Hardy spaces with $0<p^-\leq p^+<\infty$}

\authors

\author[P. D. Liu]{Peide Liu}
\address{Peide Liu \\ School of Mathematics and Statistics,
Wuhan University, 430072 Wuhan, China}
\email{pdliu@whu.edu.cn}

\author[W. Chen]{Wei Chen}
\address{Wei Chen \\ School of Mathematical Sciences,
Yangzhou University, 225002 Yangzhou, China}
\email{weichen@yzu.edu.cn}

\makeatletter
\renewcommand{\@makefntext}[1]{#1}
\makeatother \footnotetext{\noindent
Peide Liu is supported by the National Natural
Science Foundation of China(11001273).
Wei Chen is supported by the National Natural Science Foundation of China(11101353) and
the Natural Science Foundation of Jiangsu Province(BK20161326).}


%
%
\begin{abstract} We investigate the properties of the
variable Lebesgue spaces with quasi-norm on a probability space,
and give the atomic decompositions suited to the variable exponent
martingale Hardy spaces. Using the decompositions and the harmonic
mean of a variable exponent, we obtain several continuous embedding
relations between martingale Hardy spaces with small exponent.
Finally, we extend these results to the cases $0<p^-\leq
p^+<\infty.$
\end{abstract}

\maketitle

%
\section{Introduction}

Variable Lebesgue spaces $L^{p(\cdot)}(\mathbb{R}^n)$ is defined as
the set of all measurable functions $f$ such that for some
$\lambda>0$ $$\int_{\mathbb{R}^n}\left(\frac{|f(x)|}{\lambda
}\right)^{p(x)}dx<\infty,$$ where $p(\cdot): \mathbb{R}^n\rightarrow
(0,\infty)$ is a measurable function. These spaces were introduced
by Orlicz \cite{Orlicz} in 1931. The variable Lebesgue spaces, as
their name implies, are a generalization of the classical Lebesgue
spaces, replacing the constant exponent $p$ with a variable exponent
function $p(\cdot)$. The $L^{p(\cdot)}(\mathbb{R}^n)$ spaces have
many properties similar to the classical $L^p(\mathbb{R}^n)$ spaces,
but they also differ from each other in surprising and subtle ways.
For this reason the variable Lebesgue spaces have an intrinsic
interest. In addition, they are also very important for their
applications to PDEs, variational integrals with nonstandard local
growth conditions, non-Newtonian fluids and image restoration. In
the past few years the subject of variable exponent spaces has
undergone a vast development (see \cite{Cruz-Uribe01, Diening3} for
the history and references). Recently, the theory of variable
Lebesgue spaces was extended to that of variable Hardy spaces. The
variable Hardy spaces had been developed independently by Nakai and
Sawano \cite{Nakai1} and Cruz-Uribe and Daniel Wang \cite{cruz2}.
They defined atomic decompositions and proved the equivalent
definitions in terms of maximal operators. Then they gave that
singular integrals are bounded on variable Hardy spaces.

As is well known, martingale space theory is very closely related to
harmonic analysis and functional space theory. Over the past few
years, some authors have paid attention to variable exponent
martingale spaces and variable exponent martingale Hardy spaces. In
particular, Nakai and Sadasue \cite{Nakai} studied the boundedness
of Doob's maximal operator on variable exponent martingale spaces.
Jiao, Zhou, Hao and Chen \cite{Jiao} investigated the atomic
decompositions and John-Nirenberg inequalities for variable exponent
martingale Hardy spaces. In \cite{Liu1}, the famous
Burkholder-Gundy-Davis inequality for martingales and some
continuous embedding relations between martingale spaces in
classical martingale theory were extended to variable exponent Hardy
spaces. We mention that variable exponent martingale (Hardy) spaces
are very different from classical martingale spaces and function
spaces. For example, it is clear that the Jensen's inequality for
the conditional expectation is invalid. Moreover, the log-H\"{o}lder
continuity (see, e.g. \cite{Cruz-Uribe01, Diening3}) is also
invalid, because the probability space $\Omega$ has no natural
metric structure and linear structure. Thus, it is difficult but
interesting to study variable exponent martingale (Hardy) spaces.

The aim of this paper is to deal with variable exponent martingale
(Hardy) spaces. In classical martingale theory,
Burkholder-Gundy-Davis inequality holds only for any $1\leq
p<\infty,$ but some other inequalities hold for all $0<p<\infty,$ such as several inequalities for the martingales with
predictable controls (see \cite{Long, Weisz1} for more information).
Our goal is to extend the latter to the case $0<p^-\leq p^+<\infty$
and our approaches are mainly based on the atomic decompositions
suited to the variable exponent martingale Hardy spaces.

The article is organized as follows. In section \ref{sec1}, we
investigate some basic properties of variable exponent Lebegue space
$L^{p(\cdot)}$ with exponent $0<p^-\leq p^+<\infty.$ In section \ref{sec2}, we establish several
atomic decomposition theorems for the martingale Hardy spaces
$H^s_{p(\cdot)}, \mathcal{D}_{p(\cdot)}$ and
$\mathcal{Q}_{p(\cdot)}.$ In section \ref{sec3}, we discuss some properties
of harmonic mean of a variable exponent. In section \ref{sec005}, using the atomic
decompositions and the harmonic mean, we obtain several martingale
inequalities and continuous embedding relations between the spaces
with small exponents. In the last section, we extend the theorems in section \ref{sec005} to
the cases $0<p^-\leq p^+<\infty.$

Throughout this paper, $Z$ denotes the integer set. We denote by $C$
the absolute positive constant, which can vary from line to line,
and denote by $C_{p(\cdot)}$ the constant depending only on
$p(\cdot).$ The symbol $A\lesssim B$ stands for the inequality $A
\leq C B$ or $A \leq C_{p(\cdot)} B$.  If we write $A\sim B$, then
it stands for $A\lesssim B\lesssim A$.

\section{On variable Lebesgue spaces
 with quasi-norm}\label{sec1}

Let $(\Omega,\Sigma, \mu)$ be a complete probability space. We denote by
$L^0(\Omega)$ the set of all measurable functions on
$\Omega$ and $L^0_+(\Omega)$ the set of all positive members in
$L^0(\Omega).$ For $p\in L^0_+(\Omega),$ we call $p$ a variable
exponent. Let $p^-=\mathop{\hbox{ess inf}}_{\omega \in \Omega}~ p(\omega)$
and $p^+=\mathop{\hbox{ess sup}}_{\omega \in \Omega} ~p(\omega).$ If
$0<p^-\leq p^+<\infty,$ we say $p\in \mathcal{P}.$
For a variable
exponent $p,$ we define the modular of $u \in L^0(\Omega)$ by
\begin{equation}
\label{eq1.1} \rho_{p(\cdot)}(u) = E(|u|^p\chi_{\{p<\infty\}})+ \|u\chi_{\{p=\infty\}}\|_\infty,
\end{equation}
where $E$ is the expectation with respect to $\Sigma.$ Then, we denote the
variable exponent Lebesgue space by
$$L^{p(\cdot)}=\{u\in L^0(\Omega): \exists \gamma>0, \rho_{p(\cdot)}(\gamma u)
<\infty\}$$ with (quasi-)norm
\begin{equation}
\label{eq1.2} \|u\|_{p(\cdot)}=\inf \{\gamma>0: \rho_{p(\cdot)}(\frac{u}{\gamma})\leq1\}, ~~\forall u\in L^{p(\cdot) }.
\end{equation}

In the rest of this section we state some basic properties of
$L^{p(\cdot)}$ (see \cite{cruz2, Cruz-Uribe01, Diening3}). For
convenience, we give their proofs.

\begin{lemma}\label{Lem1.1} Let $p\in \mathcal{P}$ and $u\in L^{p(\cdot)}.$
\begin{enumerate}
  \item \label{Lem1.1a}$\rho_{p(\cdot)}(\lambda u)$ is continuous with respect to $\lambda$ on $(0, \infty);$
  \item \label{Lem1.1b}$\rho_{p(\cdot)}(u)<1 ~(=1, >1) $ iff $ \|u\|_{p(\cdot)}<1~(=1,>1);$
  \item \label{Lem1.1c}if $ \|u\|_{p(\cdot)}\leq1, $ then $\|u\|^{p^+}_{p(\cdot)}\leq \rho_{p(\cdot)}(u)
                       \leq\|u\|^{p^-}_{p(\cdot)};$
  \item \label{Lem1.1d}if $ \|u\|_{p(\cdot)}\geq1, $
                       then $\|u\|^{p^+}_{p(\cdot)}\geq \rho_{p(\cdot)}(u)\geq\|u\|^{p^-}_{p(\cdot)}.$
\end{enumerate}
\end{lemma}

\begin{proof}[Proof of Lemma \ref{Lem1.1}] It is clear that the function $|\lambda u|^p$ is increasing
and continuous with respect to $\lambda.$ In view of Levy's theorem and Lebesgue's dominated convergence
theorem, we get
$\rho_{p(\cdot)}$'s left-continuity and right-continuity, respectively.
Then, we have \eqref{Lem1.1a}. By the definition of $\|u\|_{p(\cdot)}$ and \eqref{Lem1.1a}, we obtain \eqref{Lem1.1b}.

To prove \eqref{Lem1.1c}, let $0<\|u\|_{p(\cdot)}\leq1.$ Then we have
$$\frac {\rho(u)}{\|u\|^{p^-}_{p(\cdot)}}
=E(\frac{|u|^p}{\|u\|^{p^-}_{p(\cdot)}})
\leq E(\frac{u}{\|u\|_{p(\cdot)}})^p
\leq E(\frac{|u|^p}{\|u\|^{p^+}_{p(\cdot)}})
=\frac{\rho(u)}{\|u\|^{p^+}_{p(\cdot)}}.$$
It follows from \eqref{Lem1.1b} that
$E(\frac{u}{\|u\|_{p(\cdot)}})^p=\rho_{p(\cdot)}(\frac{u}{\|u\|_{p(\cdot)}})=1.$
Thus $\|u\|^{p^+}_{p(\cdot)}\leq \rho_{p(\cdot)}(u)\leq\|u\|^{p^-}_{p(\cdot)}$. Similarly, we have \eqref{Lem1.1d}.
\end{proof}

\begin{lemma}\label{Lem1.2}Let $p\in \mathcal{P}, u_n, u\in L^{p(\cdot)}.$
\begin{enumerate}
  \item \label{Lem1.2a}$\sup_n\|u_n\|_{p(\cdot)}<\infty$ iff $\sup_n\rho_{p(\cdot)}(u_n)<\infty$;
  \item \label{Lem1.2b}$\|u_n-u\|_{p(\cdot)}\rightarrow0$ iff $\rho_{p(\cdot)}(u_n-u)\rightarrow0$;
  \item \label{Lem1.2c}$L^{p(\cdot)}$ is a quasi-Banach space;
  \item \label{Lem1.2d}if $p^-\geq1,$ $L^{p(\cdot)}$ is a Banach space.
\end{enumerate}
\end{lemma}

\begin{proof}[Proof of Lemma \ref{Lem1.2}] To prove \eqref{Lem1.2a}. Let $\sup_n\|u_n\|_{p(\cdot)}<\infty.$
In view of Lemma \ref{Lem1.1}\eqref{Lem1.1c} and
\eqref{Lem1.1d}, we have
$$\rho_{p(\cdot)}(u_n)\leq\left\{
  \begin{array}{ll}
    \|u_n\|^{p^-}_{p(\cdot)}, & \hbox{$\|u_n\|_{p(\cdot)}\leq 1$;} \\
    \|u_n\|^{p^+}_{p(\cdot)}, & \hbox{$\|u_n\|_{p(\cdot)}\geq 1$.}
  \end{array}
\right.$$
Then
$$\rho_{p(\cdot)}(u_n)\leq(\sup_n\|u_n\|_{p(\cdot)})^{p^-}+(\sup_n\|u_n\|_{p(\cdot)})^{p^+}<\infty.$$
Let $\rho_{p(\cdot)}(u_n)<\infty.$
In the view of Lemma \ref{Lem1.1}\eqref{Lem1.1c} and
\eqref{Lem1.1d}, we have
$$\|u_n\|_{p(\cdot)}\leq\left\{
  \begin{array}{ll}
    \rho_{p(\cdot)}(u_n)^{\frac{1}{{p^+}}}, & \hbox{$\|u_n\|_{p(\cdot)}\leq 1$;} \\
    \rho_{p(\cdot)}(u_n)^{\frac{1}{{p^-}}}, & \hbox{$\|u_n\|_{p(\cdot)}\geq 1$.}
  \end{array}
\right.$$
Then
$$\|u_n\|_{p(\cdot)}\leq(\sup_n\rho_{p(\cdot)}(u_n))^{\frac{1}{{p^+}}}+(\sup_n\rho_{p(\cdot)}(u_n))^{\frac{1}{{p^-}}}<\infty.$$

It is clear that \eqref{Lem1.2b} follows
directly from Lemma \ref{Lem1.1}\eqref{Lem1.1c}.

To check that $\|\cdot\|_{p(\cdot) }$ is a quasi-norm.
By the definition of $\|u\|_{p(\cdot)},$ we have $\|u\|_{p(\cdot)}\geq0$ and $\|u\|_{p(\cdot) }=0$ iff
$\rho_{p(\cdot)}(u)=0$ iff $u=0. $ Let $\alpha\neq0.$
It follows that
$$\|\alpha u\|_{p(\cdot)}= \inf \{\gamma >0: \rho(\frac{\alpha u}{\gamma})\leq1\}
= |\alpha|\inf \{\frac{\gamma}{|\alpha|}>0: \rho(\frac{\alpha u}{\gamma})\leq1\}
=|\alpha|\|u\|_{p(\cdot)}.$$
Now if $0<\|u\|_{p(\cdot)}< a, 0<\|v\|_{p(\cdot)}< b$ and $K\geq1,$ we deduce that
\begin{eqnarray*}
\rho_{p(\cdot)}(\frac{u+v}{K(a+b)})&=&E(\frac{(u+v)}{K(a+b)})^p
   \leq E(\frac{2}{K}\max \{|\frac{u}{a}|,|\frac{v}{b}|\})^p \\
&\leq& \frac{2^{p^+}}{K^{p^-}}( E|\frac{u}{a}|^p+E|
   \frac{v}{b}|^p) \leq  \frac{2^{{p^+}+1}}{K^{p^-}}\leq1,
\end{eqnarray*}
provided $K$ is large enough. Then $\|u+v\|_{p(\cdot)}\leq K(a+b).$ Finally we get
\begin{equation}
\label{eq1.3}\|u+v\|_{p(\cdot)}\leq K(\|u\|_{p(\cdot)}+\|v\|_{p(\cdot)}).
\end{equation}

The proof of the completeness of
$\|\cdot\|_{p(\cdot)}$ is an easy modification of
the standard one, and we omit it.

The proof of \eqref{Lem1.2d} is well known: see, for instance, \cite[Theorem 2.70]{Cruz-Uribe01}.
\end{proof}

\begin{lemma}\label{Lem1.3}
\begin{enumerate}
  \item \label{Lem1.3a} Let $p\in \mathcal{P}, u\in
L^{p(\cdot)}$ and $0<s<\infty.$ Then
\begin{equation}
\label{eq1.4}  \||u|^s\|_{p(\cdot)}=\|u\|^s_{sp(\cdot)}.
\end{equation}
  \item \label{Lem1.3b} If $r,p,q \in \mathcal{P}$ with
$\frac{1}{r}=\frac{1}{p}+\frac{1}{q},$ then there is $C=C_{p,q}>0$
such that
\begin{equation}
\label{eq1.5} \|uv\|_{r(\cdot)}\leq C\|u\|_{p(\cdot)}\|v\|_{q(\cdot)},~~ \forall u\in L^{p(\cdot)}, v\in L^{q(\cdot)}.
\end{equation}
\end{enumerate}
\end{lemma}

\begin{proof}[Proof of Lemma \ref{Lem1.3}]To prove \eqref{Lem1.3a}.
Since $p^-s>0,$ we have
\begin{eqnarray*}
\|u\|^s_{sp(\cdot)}& =& \inf \{\gamma^s>0: E|\frac{u}{\gamma}|^{sp}\leq1\} \\
&=& \inf \{\gamma^s>0: E(\frac{|u|^s}{\gamma^s})^p\leq1\}=
\||u|^s\|_{p(\cdot)}.
\end{eqnarray*}

To prove \eqref{Lem1.3b}. Without loss of generality, we suppose that
$\|u\|_{p(\cdot)}=\|v\|_{q(\cdot)}=1.$ Then we have
$$E(|u|^r)^{\frac{p}{r}}=E|u|^p=1, ~~~ E(|v|^r)^{\frac{q}{r}}=E|v|^q=1.$$
It follows that
$\||u|^r\|_{\frac{p(\cdot)}{r(\cdot)}}=\||v|^r\|_{\frac{q(\cdot)}{r(\cdot)}}=1.$
Since $\frac{p(\cdot)}{r(\cdot)}>1,
\frac{q(\cdot)}{r(\cdot)}>1$
and
$\frac{r(\cdot)}{p(\cdot)}+\frac{r(\cdot)}{q(\cdot)}=1,$ we get
$$\rho_{r(\cdot)}(uv)=E(|u||v|)^r \leq C \||u|^r\|_{\frac{p(\cdot)}{r(\cdot)}}\||v|^r\|_{\frac{q(\cdot)}{r(\cdot)}}=C$$
by H\"{o}lder's
inequality in variable exponent case(\cite[Lemma 2.2.26]{Cruz-Uribe01}).
In view of Lemma \ref{Lem1.1}\eqref{Lem1.1c}, we have $\|uv\|_{r(\cdot)}\leq
C=C\|u\|_{p(\cdot)}\|v\|_{q(\cdot)}.$
\end{proof}

\begin{lemma}\label{Lem1.4}  Let $p\in \mathcal{P}$ with $p^+\leq1$ and $
u\in L^{p(\cdot)}.$ Then

\begin{enumerate}
  \item \label{Lem1.4a}$\rho_{p(\cdot)}(u)$ is concave:
$\forall u,v\in L^{p(\cdot)}, \alpha,\beta\geq0, \alpha+\beta=1,$
$$\rho_{p(\cdot)}(\alpha u+\beta v)\geq \alpha \rho_{p(\cdot)}(u)+\beta \rho_{p(\cdot)}(v).$$

 \item \label{Lem1.4b}if $ \|u\|_{p(\cdot)}\leq1, $ then
$\|u\|_{p(\cdot)}\leq\rho_{p(\cdot)}(u);$
if $\|u\|_{p(\cdot)}\geq1,$ then $\|u\|_{p(\cdot)}\geq\rho_{p(\cdot)}(u).$
\end{enumerate}
\end{lemma}

\begin{proof}[Proof of Lemma \ref{Lem1.4}]For any $\alpha, \beta, u, v,$ since $p\leq1,$
we have $|\alpha u+\beta v|^p\geq \alpha |u|^p+\beta |v|^p.$
It follows that $$\rho_{p(\cdot)}(\alpha u+\beta v)\geq \alpha \rho_{p(\cdot)}(u)+\beta \rho_{p(\cdot)}(v).$$
In view of Lemma \ref{Lem1.1}\eqref{Lem1.1c}, we get \eqref{Lem1.4b}.
\end{proof}

In this paper, we shall use the following theorem many times, which
is known as Aoki-Rolewicz's theorem.
\begin{theorem}\cite[Theorem 1.3]{Mitrea}\label{lem00} Let $X$ be a vector space equipped with a quasinorm
$\|\cdot\|.$ Then there exists a quasinorm $\|\cdot\|_*$ on X that
is equivalent to $\|\cdot\|$ and is a $\eta-$norm for some $0<\eta\leq1,$
i.e., it satisfies
$\|x+y\|^\eta_*\leq\|x\|^\eta_*+\|y\|^\eta_*$
 for all $x,y\in X.$
\end{theorem}

\begin{remark}\label{rk01}
If $p^+<1,$ denote $\||\cdot|\|_{p(\cdot)}=\|\cdot\|_*^\eta,$ then
we have
$\||\cdot|\|_{p(\cdot)}\thickapprox\|\cdot\|^\eta_{p(\cdot)}$ and
\begin{equation}
\label{eq1.6} \||u+v|\|_{p(\cdot)}\leq\||u|\|_{p(\cdot)}+\||v|\|_{p(\cdot)},~\forall ~u,~v\in L^{p(\cdot)}.
\end{equation}
\end{remark}

\section{The atomic decompositions}\label{sec2}

Let $(\Sigma_n)_{n\geq 0}$ be a nondecreasing sequence of
sub-$\sigma$-algebras of $\Sigma$ with $\Sigma=\bigvee\Sigma_n$ and
$f=(f_n)_{n\geq0}$ a martingale adapted to $(\Sigma_n)_{n\geq0}$
with its difference sequence $(d_nf)_{n\geq0}, $ where
$d_nf=f_n-f_{n-1}$ (with convention $f_{-1}=0, \Sigma_{-1}=\{\Omega,\emptyset\}).$
We denote by $E_n$ the conditional
expectation with respect to $\Sigma_n.$ For a martingale
$f=(f_n)_{n\geq0},$ we define its maximal function $f^*$, square
function $S(f)$ and conditional square function $s(f)$ as usual. The
variable exponent Hardy spaces of martingales are defined as
follows:
\begin{eqnarray*}
  && H^*_{p(\cdot)}=\{f=(f_n): \|f\|_{H^*_{p(\cdot)}}=\|f^*\|_{p(\cdot)}<\infty\},\\
  && H^S_{p(\cdot)}=\{f=(f_n): \|f\|_{H^S_{p(\cdot)}}=\|S(f)\|_{p(\cdot)}<\infty\},\\
  && H^s_{p(\cdot)}=\{f=(f_n): \|f\|_{H^s_{p(\cdot)}}=\|s(f)\|_{p(\cdot)}<\infty\}.
\end{eqnarray*}
Let $\Lambda$ be the class of non-negative, non-decreasing and adapted sequence $\lambda=(\lambda_n)$ with $\lambda_\infty=\lim_{n\rightarrow\infty}\lambda_n.$ We define
\begin{eqnarray*}
&&\mathcal{Q}_{p(\cdot)}=\{f=(f_n): \exists \lambda\in\Lambda, S_n(f)\leq \lambda_{n-1},\|f\|_{\mathcal{Q}_{p(\cdot)}}
    =\inf_{\lambda\in\Lambda}\|\lambda_\infty\|_{p(\cdot)}<\infty\},\\
&&\mathcal{D}_{p(\cdot)}=\{f=(f_n): ~ \exists \lambda\in\Lambda, ~|f_n|\leq
    \lambda_{n-1}, \|f\|_{\mathcal{D}_{p(\cdot)}}
    = \inf_{\lambda\in\Lambda}\|\lambda_\infty\|_{p(\cdot)}<\infty\}.
\end{eqnarray*}
Using the method of \cite[P.51]{Long}, we can construct a
$\lambda\in\Lambda$ such that
$$\|f\|_{\mathcal Q_{p(\cdot)}}
=\|\lambda_\infty\|_{p(\cdot)}\hbox{ or }\|f\|_{\mathcal{D}_{p(\cdot)}}
=\|\lambda_\infty\|_{p(\cdot)},$$
which is called $f$'s optimal predictable
control in $\mathcal Q_{p(\cdot)}$ or $\mathcal{D}_{p(\cdot)},$
respectively.

Let $p\in \mathcal{P}.$ We say that a $\Sigma-$measurable
function $a$ is an atom of first category, if there exists a stopping
time $\tau$ such that
\begin{enumerate}
  \item \label{atom1} $E_n a=0, $ when $n\leq \tau;$
  \item \label{atom2} $\|s(a)\|_\infty\leq\|\chi_{\{\tau<\infty\}}\|_{p(\cdot)}^{-1}.$
\end{enumerate}
$a$ is said to be an atom of
second or third category if \eqref{atom2} holds when we use $S(a)$
or $a^*$ instead of $s(a)$, respectively. We denote by $p-A_1, p-A_2$ or $p-A_3$ the sets
of all atoms of first, second, or third category, respectively.

The atomic decomposition of Hardy spaces of functions defined on
$R^n$ is due to Coifman \cite{Coifman1}(see \cite{Coifman} for more information).
Meanwhile, Herz \cite{Herz} introduced the atomic decomposition to martingale
theory. Then Bernard and Maisonneuve \cite{Bernard}, Chevalier
\cite{Chevalier}, Weisz \cite{Weisz} used them to studied martingale
space theory (see \cite{Weisz1} for more information). In this
section, we establish some atomic decomposition theorems for the
martingales in $H^s_{p(\cdot) }, \mathcal{Q}_{p(\cdot)}$ and
$\mathcal{D}_{p(\cdot)},$ respectively.

\begin{theorem}\label{Thm2.1} Let $p\in \mathcal{P}$ and $f\in
H^s_{p(\cdot) }.$ Then there exist a sequence $(a^k, k\in Z)$ of
$p-A_1$ atoms with $\|a^k\|_{H^s_{p(\cdot)}}\leq1,$ and a sequence
$(\theta_k, k\in Z)$ of nonnegative numbers such that
\begin{equation}
\label{eq2.1} f_n=\sum_{k\in Z}\theta_k E_na^k, ~~a.e. ~~\forall
n\geq 0
\end{equation}
and
\begin{equation}
\label{eq2.2} C^{-1} (\sum_{k\in Z}\theta_k^{p^+})^{\frac{1}{p^+}}
\leq \|f\|_{H^s_{p(\cdot)}}\leq C  (\sum_{k\in Z}\theta_k^{p^-})^{\frac{1}{p^-}},
\end{equation}
where $C$ is a constant independent of $f.$

Moreover, the sum $\sum_{k=j}^m\theta_k a^k$ converges to $f$ in
$H^s_{p(\cdot) }$ as $j\rightarrow -\infty, m\rightarrow\infty.$
\end{theorem}

\begin{proof}[Proof of Theorem \ref{Thm2.1}]
For $f=(f_n)\in H^s_{p(\cdot)}$ and $k\in Z,$ define
stopping times
\begin{equation}
\label{eq2.3} \tau_k=\inf \{n: s_{n+1}(f)>2^k\}, ~~ (\inf \emptyset=\infty)
\end{equation}
and $\theta_k=3\cdot2^k \|\chi_{A_k}\|_{p(\cdot)},$ where
$A_k=\{\tau_k<\infty\}=\{s(f)>2^k\}.$
Let
$$a^k_n=\theta_k^{-1}(f^{\tau_{k+1}}_n-f^{\tau_k}_n),~~~ a^k=(a^k_n)_{n=0}^\infty.$$
Then $da^k_n=\theta_k^{-1}(df_{\tau_{k+1}\wedge n}-df_{\tau_k\wedge n}).$
It follows that $a^k$ is a martingale with
\begin{equation}
\label{eq2.04}E_n a^k=0, ~\forall n\leq \tau_k
\end{equation}
and
\begin{equation}
\label{eq2.4} \sum_{k\in Z}\theta_ka^k_n=\sum_{k\in Z}(f_{\tau_{k+1}\wedge n}-f_{\tau_k\wedge n})=f_n.
\end{equation}
Moreover
$$s(a^k)\leq\theta_k^{-1}(s_{\tau_{k+1}}(f)+s_{\tau_k}(f))
\leq\theta_k^{-1}(3\cdot2^k)
\leq \|\chi_{A_k}\|_{p(\cdot)}^{-1}.
$$ Thus
\begin{equation}
\label{eq2.5}\|s(a^k)\|_\infty \leq \|\chi_{A_k}\|_{p(\cdot)}^{-1}, ~~~ \forall k\in Z.
\end{equation}
By classical Burkholder-Gundy-Davis inequality, $a_n^k$ converges to
a function a.e. and in $L_2.$ We still denote the function by
$a^k.$ Then $ a^k\in L_2$ and $a^k_n=E_na^k.$ It follows from \eqref{eq2.04} and \eqref{eq2.5} that every $a^k$ is a
$p-A_1$ atom and \eqref{eq2.1} holds.
In addition, using \eqref{eq2.5}, we get
$$\rho_{p(\cdot)}(s(a^k))=Es(a^k)^p\chi_{A_k}
=E[s(a^k)\|\chi_{A_k}\|_{p(\cdot)}]^p[\frac{\chi_{A_k}}{\|\chi_{A_k}\|_{p(\cdot)}}]^p
\leq E[\frac{\chi_{A_k}}{\|\chi_{A_k}\|_{p(\cdot)}}]^p =1,$$
Then, by Lemma \ref{Lem1.1}\eqref{Lem1.1b}, we have
$$ \|a^k\|_{H^s_{p(\cdot)}}=\|s(a^k)\|_{p(\cdot)}\leq1, ~~~ \forall k\in Z.$$

To estimate \eqref{eq2.2}, we split it into three steps.

$Step~1.$  We define the function $g=\sum_{k\in Z} 3\cdot
2^k\chi_{A_k}$ and show $\|g\|_{p(\cdot) }\sim
\|f\|_{H^s_{p(\cdot)}}.$ It is clear that
$$g= 3\sum_{k\in Z}2^k\chi_{A_k}=3\sum_{k\in Z}(2^{k+1}-2^k)\chi_{A_k}
=6\sum_{k\in Z} 2^k\chi_{A_k\setminus A_{k+1}}.$$
For any $\gamma>0,$
we have
\begin{eqnarray*}
\int_\Omega(\frac{\sum_{k\in Z}2^k\chi_{A_k \backslash
A_{k+1}}}{\gamma})^p  d\mu
    &=&\sum_{k\in Z}\int_{A_k \backslash A_{k+1}}(\frac{2^k}{\gamma})^pd\mu \\
&\leq& \sum_{k\in Z}\int_{A_k \backslash
A_{k+1}}(\frac{s(f)}{\gamma})^p  d\mu
    \leq\int_\Omega(\frac{s(f)}{\gamma})^p  d\mu.
\end{eqnarray*}
Taking
$\gamma=\|s(f)\|_{p(\cdot) }$
and noticing
$\int_\Omega(\frac{s(f)}{\gamma})^p d\mu
=\rho_{p(\cdot)}(\frac{s(f)}{\|s(f)\|_{p(\cdot)}})=1,$
we get
$$\| \sum_{k\in Z}2^k \chi_{A_k \backslash A_{k+1}}\|_{p(\cdot) }\leq
\gamma.$$
Hence $\|g\|_{p(\cdot) }\leq 6 \|f\|_{H^s_{p(\cdot) }}.$
On the other hand, for any $\gamma>0,$ we have
$$\int_\Omega(\frac{\sum_{k\in Z}2^{k+1}\chi_{A_k \backslash A_k}}{\gamma})^p  d\mu
=\sum_{k\in Z}\int_{A_k \backslash A_{k+1}}(\frac{2^{k+1}}{\gamma})^p  d\mu
\geq\int_\Omega(\frac{s(f)}{\gamma})^p  d\mu,$$
Taking $\gamma=\|s(f)\|_{p(\cdot)},$ we get
$$\| \sum_{k\in Z}2^{k+1} \chi_{A_k \backslash A_{k+1}}\|_{p(\cdot)}\geq\gamma.$$
Hence $\|g\|_{p(\cdot)}\geq 3 \|f\|_{H^s_{p(\cdot)}}.$

$Step~2.$ To estimate the first inequality in \eqref{eq2.2}. For any
$\gamma>0,$ we have
\begin{eqnarray*}
E(\frac{g}{\gamma})^p &=& E(\sum_{k\in
Z}\frac{6\cdot2^k\chi_{A_k\setminus
      A_{k+1}}}{\gamma})^p = E\sum_{k\in Z}(\frac{6\cdot2^k\chi_{A_k\setminus A_{k+1}}}{\gamma})^p \\
&=&E\sum_{k\in Z}(2^p-1)(\frac{3 \cdot 2^k\chi_{A_k}}{\gamma})^p
      \geq (2^{p^-}-1)E \sum_{k\in Z}(\frac{\theta_k\chi_{A_k}}{\gamma \|\chi_{A_k}\|_{p(\cdot)}})^p.
\end{eqnarray*}
Taking $\gamma=(\sum_{k\in Z}\theta_k^{p^+})^\frac{1}{p^+}$ and denoting $C=2^{p^-}-1,$ we get
$$E(\frac{g}{\gamma})^p
\geq C \sum_{k\in Z}(\frac{\theta_k}{\gamma})^{p^+}E(\frac{\chi_{A_k}}{\|\chi_{A_k}\|_{p(\cdot) }})^p
= C \sum_{k\in Z}(\frac{\theta_k}{\gamma})^{p^+} = C.$$
Lemma \ref{Lem1.1}\eqref{Lem1.1c} shows that
$\|\frac{g}{\gamma}\|_{p(\cdot) }\geq C.$  It follows that
$\|g\|_{p(\cdot) }\geq C\gamma
= C (\sum_{k\in Z}\theta_k^{p^+})^\frac{1}{p^+}.$

$Step~3.$ To estimate the second inequality in \eqref{eq2.2}. For
any $\gamma>0,$ we have
$$E(\frac{g}{\gamma})^p
= E\sum_{k\in Z}(2^p-1)(\frac{3 \cdot 2^k\chi_{A_k}}{\gamma})^p
\leq (2^{p^+}-1)E \sum_{k\in Z}(\frac{\theta_k\chi_{A_k}}{\gamma \|\chi_{A_k}\|_{p(\cdot)}})^p.
$$
If $\sum_{k\in Z}\theta_k^{p^-}<\infty$, taking $\gamma=(\sum_{k\in
Z}\theta_k^{p^-})^\frac{1}{p^-}$ (the inequality naturally holds
when $\sum_{k\in Z}\theta_k^{p^-}=\infty$) and denoting
$C=2^{p^+}-1,$ we get
$$E(\frac{g}{\gamma})^p
\leq C \sum_{k\in Z}(\frac{\theta_k}{\gamma})^{p^-}E(\frac{\chi_{A_k}}{\|\chi_{A_k}\|_{p(\cdot) }})^p
= C \sum_{k\in Z}(\frac{\theta_k}{\gamma})^{p^-} = C.$$
Therefore,
$$\|g\|_{p(\cdot) }\leq C\gamma
= C (\sum_{k\in Z}\theta_k^{p^-})^\frac{1}{p^-}.$$

Now we prove that the sum $\sum_{k=j}^m\theta_k a^k$ converges to
$f$ in $H^s_{p(\cdot)}.$ It follows from \eqref{eq2.1} that $f=
\sum_{k\in Z}\theta_ka^k$ and $f-\sum_{k=j}^m\theta_k
a^k=f-f^{\tau_{m+1}}+f^{\tau_j}.$ Because of $s(f^{\tau_j})\leq
2^j,$ we have $\|f^{\tau_j}\|_{p(\cdot) }\leq2^j\rightarrow 0$ as
$j\rightarrow -\infty.$ In addition, we have
$$s(f-f^{\tau_{m+1}})^p \leq s(f)^p
\hbox{ and } s(f-f^{\tau_{m+1}})^p \rightarrow0.$$
In view of
Lebesgue's dominated convergence, we have
$\rho_{p(\cdot)}(s(f-f^{\tau_{m+1}}))\rightarrow0$ as $m\rightarrow\infty.$ It follows that
$\|s(f-f^{\tau_{m+1}})\|_{p(\cdot) }\rightarrow0$
as $m\rightarrow\infty$. Finally, the sum $\sum_{k=j}^m\theta_k a^k$ converges to $f$ in
$H^s_{p(\cdot) }$ as $j\rightarrow -\infty, m\rightarrow\infty.$
\end{proof}

\begin{theorem}\label{Thm2.2} Let $p\in \mathcal{P}$ and $f\in
\mathcal{Q}_{p(\cdot)} (\mathcal{D}_{p(\cdot)}).$ Then there exist a
sequence $(a^k, k\in Z)$ of $p-A_2 (p-A_3)$ atoms and a sequence
$(\theta_k, k\in Z)$ of nonnegative numbers such that
\begin{equation}
\label{eq2.6} f_n=\sum_{k\in Z}\theta_k E_na^k,~~a.e. ~~ \forall
n\geq 0
\end{equation}
and
\begin{equation}
\label{eq2.7} C^{-1}(\sum_{k\in Z}\theta_k^{p^+})^{\frac{1}{p^+}}
\leq \|f\|_{\mathcal{Q}_{p(\cdot)}}(\|f\|_{\mathcal{D}_{p(\cdot)}})
\leq C(\sum_{k\in Z}\theta_k^{p^-})^{\frac{1}{p^-}},
\end{equation}
where $C$ is a
constant independent of $f.$

Moreover, if $p^+\leq\eta,$ then $\sum_{k=j}^m\theta_k a^k$ converges to $f$ in
$\mathcal{Q}_{p(\cdot)}(\mathcal{D}_{p(\cdot)})'$s quasi-norm, as
$j\rightarrow -\infty, m\rightarrow\infty.$
\end{theorem}

\begin{proof}[Proof of Theorem \ref{Thm2.2}] The proof is similar to that of Theorem \ref{Thm2.1}. Here we only
give a outline of the proof for $\mathcal{Q}_{p(\cdot)}.$

Let $f\in \mathcal{Q}_{p(\cdot)}$ and $\lambda=(\lambda_n)$ be
$S_n(f)$'s optimal predictable control. We define $$\tau_k=\inf \{n:
\lambda_n> 2^k\},~~ A_k=\{\tau_k<\infty\}=\{\lambda_\infty>2^k\}$$
and $\theta_k, a^k_n, a^k$ as in the proof of Theorem \ref{Thm2.1}.
Instead of estimating $s(f),$ we estimate
$\lambda_\infty$. We have
$$\|\lambda_\infty\|_\infty \leq \|\chi_{A_k}\|_{p(\cdot)}^{-1}, ~~
\|a^k\|_{\mathcal{Q}_{p(\cdot)}}\leq1$$
and \eqref{eq2.6} holds.
To estimate $\|\lambda_\infty\|_{p(\cdot)},$
we define $g$ as in the proof of Theorem \ref{Thm2.1}. Then similar
computations show that $\|\lambda_\infty\|_{p(\cdot)} \sim
\|g\|_{p(\cdot)}$ and \eqref{eq2.7} holds.

It remains to prove the last statement. It is clear that
$$S(f-f^{\tau_{m+1}}) \leq \sum_{k\geq m+1}S(f^{\tau_{k+1}}-f^{\tau_k}) =\sum_{k\geq m+1}\theta_k S(a^k).$$
Set $\rho^k_n=\|\lambda_\infty\|_{p(\cdot)}\chi_{\{\tau_k\leq n\}},~
\rho_n= \sum_{k=m+1}^\infty\theta_k\rho^k_n,$ then $\rho^k_n$ is
adapted and
$$S_n(f-f^{\tau_{m+1}}) \leq \sum_{k\geq m+1}S(f^{\tau_{k+1}}_n-f^{\tau_k}_n)
\leq \sum_{k\geq m+1}\theta_k S_n(a^k) \leq \rho_{n-1}.$$
It follows that $(\rho_n,n\geq0)$ is a predictable control of $(S_n(f-f^{\tau_{m+1}}), n\geq0).$
In addition, we have
\begin{eqnarray*}
\|f-f^{\tau_{m+1}}\|^\eta_{\mathcal{Q}_{p(\cdot)}} &\leq &
 \|\rho_\infty\|^\eta_{p(\cdot)} \leq  C \sum_{k\geq
 m+1}\|\theta_k\lambda_\infty\chi_{A_k}\|^\eta_{p(\cdot)} \\
                 &\leq & C \sum_{k\geq m+1}\theta^\eta_k
        \leq C(\sum_{k\geq m+1}\theta_k^{p^+})^{\frac{\eta}{p^+}}.
\end{eqnarray*}
Thus, $\|f-f^{\tau_{m+1}}\|_{\mathcal{Q}_{p(\cdot)}}\rightarrow 0$ in
as $m\rightarrow +\infty.$ Moreover, because of
$S(f^{\tau_j})\leq 2^j,$ we have that
$\|f^{\tau_j}\|_{\mathcal{Q}_{p(\cdot)}}\rightarrow 0$ as
$j\rightarrow -\infty.$ Finally, $\sum_{k=j}^m\theta_k a^k$
converges to $f$ in $\mathcal{Q}_{p(\cdot)}$, as $j\rightarrow
-\infty, m\rightarrow\infty.$
\end{proof}

\section{On harmonic mean of variable exponent}\label{sec3}

In this section we define the harmonic mean and the averaging
operator. Using the method of proving \cite[Lemma 4.5.3.]{Diening3}, we estimate $\|\chi_A\|_{p(\cdot)},$ where $A\in \Sigma$ and $p\in \mathcal{P}.$
In what follows, $\mu(A)$ is denoted by $|A|,$ $\forall A\in\Sigma.$

\begin{definition}\cite{Diening3}\label{def001}
Let $p\in L_+^0(\Omega)$ and
$A\in \Sigma$ with $|A|>0.$ We define the harmonic mean of $p$ by setting
\begin{equation*}
\label{eq3.1} \frac{1}{p_A}=\frac{1}{|A|}\int_A\frac{1}{p}d\mu.
\end{equation*}
\end{definition}

By Definition \ref{def001},  we immediately deduce Lemma \ref{Lem3.1}.

\begin{lemma}\label{Lem3.1} Let $p,q,r \in L_+^0(\Omega).$
\begin{enumerate}
  \item \label{Lem3.1a}If $s>0,$ then $(sp)_A=sp_A.$
  \item \label{Lem3.1b}If $\frac{1}{r}=\frac{1}{p}+\frac{1}{q},$ then
$\frac{1}{r_A}=\frac{1}{p_A}+\frac{1}{q_A}.$
\end{enumerate}
\end{lemma}

\begin{definition} Let $A\in \Sigma$ with $|A|>0.$ We define the averaging
operator $T_A: L^1\rightarrow L^0$ by setting
$T_Au=\frac{1}{|A|}\int_A u d\mu  \chi_A.$
\end{definition}

\begin{lemma}\label{Lem3.2} If $1\leq p^-\leq p^+<\infty,$ the operator
$T_A: L^{p(\cdot)}\rightarrow L^{p(\cdot)}$ is uniformly bounded
with respect to $A\in \Sigma.$
\end{lemma}

\begin{proof}[Proof of Lemma \ref{Lem3.2}]
Recall that there is a constant $C>0$ such that
\begin{equation}\label{eq3.2}
\sup_{\lambda>0}\|\lambda \chi_{\{f^*>\lambda\}}\|_{p(\cdot)}\leq C\|f\|_{p(\cdot)}, ~\forall f=(f_n),
\end{equation}
where $f^*$ is Doob's maximal operator (see \cite{Jiao}).

For $u\in L^{p(\cdot)}$ and $A\in \Sigma,$ let $f=T_Au$ and $\lambda=\frac{1}{2|A|}\int_A u d\mu.$
Using \eqref{eq3.2}, we obtain
$$\|T_Au\|_{p(\cdot)}=\|2\lambda \chi_A\|_{p(\cdot)}
\leq 2\|\lambda \chi_{\{f^*>\lambda\}}\|_{p(\cdot)}
\leq C \|u\|_{p(\cdot)}.$$
\end{proof}

\begin{theorem} \label{Thm3.3} Let $p \in L_+^0(\Omega)$ with $p^->0$ and
$A\in \Sigma$ with $|A|>0,$ then
\begin{equation}\label{eq3.3}
\|\chi_A\|_{p(\cdot)}\sim|A|^{\frac{1}{p_A}}.
\end{equation}
\end{theorem}

\begin{proof}[Proof of Theorem \ref{Thm3.3}] We split the proof into two steps.

$Step ~1.$ Let $p\geq1.$ Setting $\varphi_p(t)=t^p,$ we regard
$\varphi_p(t)$ as a function of two variables $p$ and $t.$ For a
fixed $t\geq0,$ $\varphi_p(t)$ is a convex function of $p.$ Defining
$$\varphi_p^{-1}(t)=\inf \{s\geq0, ~ \varphi_p(s)\geq t\},$$ we have
$\varphi^{-1}_p(t)=t^{\frac{1}{p}}.$ Let $p'$ be $p$'s conjugate exponent. Setting
$\varphi_{p'}(t)=t^{p'},$ we have
$\varphi_{p'}^{-1}(t)=t^{\frac{1}{p'}}$ with
$t^\infty=\infty\chi_{(1,\infty)}(t)$ for given $t\geq0.$ It follows
that
\begin{equation}\label{eq3.4}
\varphi_p(\varphi_p^{-1}(t))\leq t, ~~
\varphi^{-1}_p(t)\varphi_p^{-1}(\frac{1}{t})=1, ~~
\varphi^{-1}_p(t)\varphi_{p'}^{-1}(t)\geq t.
\end{equation}
For $t>0$ and $|A|>0,$ we have
\begin{equation}\label{eq3.5}\frac{1}{|A|}\int_A\varphi_p^{-1}(t)d\mu\geq \frac{t}{|A|}\int_A\frac{1}{\varphi_{p'}^{-1}(t)}d\mu
\geq
\frac{t}{\frac{1}{|A|}\int_A\varphi_{p'}^{-1}(t)d\mu},\end{equation}
where we have used Jensen's inequality and the
convexity of $z\rightarrow \frac{1}{z}.$

Let $u =\varphi_p^{-1}(\frac{1}{|A|})\chi_A, ~v
=\varphi_{p'}^{-1}(\frac{1}{|A|})\chi_A. $ It is easy to check that
$\rho_p(u)\leq1, \rho_{p'}(v)\leq1.$ Then
$\|u\|_{p(\cdot)}\leq1, \|v\|_{p'(\cdot)}\leq1.$
Applying \eqref{eq1.5}, we have
$$\int_A\varphi_{p'}^{-1}(\frac{1}{|A|})d\mu
= \int_A v d\mu\leq C\|\chi_A\|_{p(\cdot)}\|v\|_{p'(\cdot)}
\leq C\|\chi_A\|_{p(\cdot)}.$$
Let $t=|A|.$ Then it follows from \eqref{eq3.4} and the convexity of
$\varphi_{p'}^{-1}(\frac{1}{t})$ that
\begin{eqnarray}
 |A|^{\frac{1}{p_A}} &=& \varphi_{p_A}^{-1}(|A|) \leq
\frac{1}{\varphi_{p_A}^{-1}(\frac{1}{|A|})}\leq
|A|\varphi_{p'_A}^{-1}(\frac{1}{|A|})\nonumber \\
& \leq & \int_A\varphi_{p'}^{-1}(\frac{1}{|A|})d\mu\leq
C\|\chi_A\|_{p(\cdot)}.\label{eq3.6}
\end{eqnarray}
On the other hand, using Lemma \ref{Lem3.2},
we have
\begin{eqnarray*}
 \|\chi_A\|_{p(\cdot)}\frac{1}{|A|}\int_A\varphi_p^{-1}(\frac{1}{|A|})d\mu
     = \|\frac{1}{|A|} \int_A u d\mu \chi_A \|_{p(\cdot)}
= \|T_A u\|_{p(\cdot)}\leq C \|u\|_{p(\cdot)}\leq C.
\end{eqnarray*}
It follows from \eqref{eq3.4} and the convexity
of $\varphi_{p}^{-1}(\frac{1}{t})$ that
\begin{equation}\label{eq3.7}
\|\chi_A\|_{p(\cdot)}\leq \frac{C}{\frac{1}{|A|}\int_A\varphi_p^{-1}(\frac{1}{|A|})d\mu }
\leq \frac{C}{\varphi_{p_A}^{-1}(\frac{1}{|A|})}
\leq C \varphi_{p_A}^{-1}(|A|)=C |A|^{\frac{1}{p_A}}.
\end{equation}
Combining this with \eqref{eq3.6}, we obtain \eqref{eq3.3} for $p\geq1.$

$Step ~2.$ For $p \in L_+^0(\Omega)$ with $p^->0,$ let
 $q=p/p^-.$ By Step 1, we have $\|\chi_A\|_{q(\cdot)}\sim |A|^{\frac{1}{q_A}}.$ It follows from
Lemmas \ref{Lem3.1}\eqref{Lem3.1a} and \ref{Lem1.3}\eqref{Lem1.3a} that
$$|A|^{\frac{1}{p_A}}=|A|^{\frac{1}{p^-q_A}}
=(|A|^{\frac{1}{q_A}})^{\frac{1}{p^-}}\sim
\|\chi_A\|^{\frac{1}{p^-}}_{q(\cdot)}
=\|\chi_A\|^{1/{p^-}}_{p(\cdot)/p^-}=\|\chi_A\|_{p(\cdot)}.$$
\end{proof}

\begin{theorem} \label{Thm3.4}Let $p,q,r \in  L_+^0(\Omega).$
\begin{enumerate}
  \item \label{Thm3.4a} If $p\geq1$, then
\begin{equation}\label{eq3.8} \|\chi_A\|_{p(\cdot)}\|\chi_A\|_{p'(\cdot)}\sim |A|,\end{equation}
where $p'$ is $p$'s conjugate exponent;
  \item \label{Thm3.4b} If $\frac{1}{r}=\frac{1}{p}+\frac{1}{q}$ with $r^->0,$ then
\begin{equation}\label{eq3.9}
\|\chi_A\|_{r(\cdot)}\sim
\|\chi_A\|_{p(\cdot)}\|\chi_A\|_{q(\cdot)}.
\end{equation}
\end{enumerate}
\end{theorem}

\begin{proof}[Proof of Theorem \ref{Thm3.4}] Using Lemmas \ref{Lem3.2} and \ref{Lem3.1}, we have
$$\|\chi_A\|_{p(\cdot)}\|\chi_A\|_{p'(\cdot)}\sim |A|^{\frac{1}{p_A}}|A|^{\frac{1}{p'_A}}=|A|^{\frac{1}{p_A}+\frac{1}{p'_A}}=|A|$$
and
$$\|\chi_A\|_{r(\cdot)}\sim |A|^{\frac{1}{r_A}}
= |A|^{\frac{1}{p_A}+\frac{1}{q_A}}
=|A|^{\frac{1}{p_A}}|A|^{\frac{1}{q_A}}\sim \|\chi_A\|_{p(\cdot)}\|\chi_A\|_{q(\cdot)}.$$
\end{proof}

\section{Some inequalities for spaces
with $0<p^-\leq p^+\leq \eta$}\label{sec005}

In this section, using the atomic decompositions and the harmonic
mean, we prove several martingale inequalities of the Hardy spaces
with small exponents $0<p^-\leq p^+\leq \eta,$ where
$\eta$ is the constant in Theorem \ref{lem00}.

\begin{theorem}\label{Thm4.1} If $p\in \mathcal{P}$ with $p^+\leq \eta,$ then
\begin{equation}\label{eq4.1}
\|f\|_{H^*_{p(\cdot)}}\lesssim \|f\|_{H^s_{p(\cdot)}},
\|f\|_{H^S_{p(\cdot)}}\lesssim \|f\|_{H^s_{p(\cdot)}},
\forall f=(f_n).
\end{equation}
In particular,
$$\|S^2(f)\|_{p(\cdot)}\lesssim  \|s^2(f)\|_{p(\cdot)}, \forall f=(f_n). $$
\end{theorem}

\begin{proof}[Proof of Theorem \ref{Thm4.1}] Let $f\in H^s_{p(\cdot)}.$ It follows from Theorem \ref{Thm2.1} that $f$
has an atomic decomposition satisfying \eqref{eq2.1} and \eqref{eq2.2}.

For every $p-A_1$ atom $a^k,$ we estimate the quasi-norms of
$a^{k*}$ and $S(a^k)$ in $L^{p(\cdot)}.$ For $p$ with
$\frac{1}{p}=\frac{1}{2}+\frac{1}{q},$ we have $p\leq q\leq2.$ Using \eqref{eq1.5},
\eqref{eq2.5}, Theorem \ref{Thm3.4}\eqref{Thm3.4b} and classical
Burkholder-Gundy-Davis inequality, we have
\begin{eqnarray}
\|a^{k*}\|_{p(\cdot)}& =& \|a^{k*}\chi_{A_k}\|_{p(\cdot)}
        \leq C\|a^{k*}\|_2\|\chi_{A_k}\|_{q(\cdot)}\nonumber\\
&\leq & C\|S(a^k)\|_2\|\chi_{A_k}\|_{q(\cdot)}=C\|s(a^k)\|_2\|\chi_{A_k}\|_{q(\cdot)}\nonumber\\
&\leq&\frac{C\|\chi_{A_k}\|_2\|\chi_{A_k}\|_{q(\cdot)}}{\|\chi_{A_k}\|_{p(\cdot)}}\leq
C. \label{eq4.2}
\end{eqnarray}
Similarly,
\begin{eqnarray}
\|S(a^k)\|_{p(\cdot)} &=& \|S(a^k)\chi_{A_k}\|_{p(\cdot)}\leq C\|S(a^k)\|_2\|\chi_{A_k}\|_{q(\cdot)}\nonumber\\
&=& C\|s(a^k)\|_2\|\chi_{A_k}\|_{q(\cdot)}\nonumber\\
&\leq&\frac{C\|\chi_{A_k}\|_2\|\chi_{A_k}\|_{q(\cdot)}}{\|\chi_{A_k}\|_{p(\cdot)}}\leq C.\label{eq4.3}
\end{eqnarray}

It follows from \eqref{eq2.1} and \eqref{eq1.6} that $f^*\leq
\sum_{k\in Z}\theta_k a^{k*}$ and
\begin{eqnarray*}
\|f^*\|^\eta_{p(\cdot)}&\leq & \sum_{k\in Z} \|\theta_k a^{k*}\|^\eta_{p(\cdot)}\leq \sum_{k\in Z} \theta_k^\eta\| a^{k*}\|^\eta_{p(\cdot)}\\
                       &\leq& C\sum_{k\in Z}\theta_k^\eta
                       \leq C(\sum_{k\in Z}\theta_k^{p^+})^\frac{\eta}{p^+}.
\end{eqnarray*}
Combining this with \eqref{eq2.2}, we get
$$\|f^*\|_{p(\cdot)}\leq C\| f \|_{H^s_{p(\cdot)}}.$$ Similarly, we have
$$\|S(f)\|_{p(\cdot)}\leq C \|f\|_{H^s_{p(\cdot)}}.$$

Taking $p/2$ instead of $p$ in the last inequality and using Lemma
\ref{Lem1.3}\eqref{Lem1.3a}, we get 
\begin{equation*}\|S^2(f)\|_{p(\cdot)}\leq C
\|s^2(f)\|_{p(\cdot)}.\end{equation*}
\end{proof}

\begin{theorem} \label{Thm4.2}If $p\in \mathcal{P}$ with $p^+\leq\eta,$ then
\begin{equation}\label{eq4.4}\| f \|_{H^*_{p(\cdot)}} \lesssim  \| f \|_{\mathcal{Q}_{p(\cdot)}},
~~\| f \|_{H^S_{p(\cdot)}} \lesssim  \|f\|_{\mathcal{D}_{p(\cdot)}},
~~\forall f=(f_n).
\end{equation}\end{theorem}

\begin{proof}[Proof of Theorem \ref{Thm4.2}] Let $f\in
\mathcal{Q}_{p(\cdot)}.$ It follows from Theorem \ref{Thm2.2} that $f$ has an atomic decomposition satisfying
\eqref{eq2.6} and \eqref{eq2.7}. For every $p-A_2$ atom
$a^k,$ we estimate the quasi-norm of
$a^{k*}$ in $L^{p(\cdot)}.$ For $p$ with
$\frac{1}{p}=\frac{1}{2}+\frac{1}{q},$ we have $p\leq q\leq2.$
Using \eqref{eq1.5}, Theorem \ref{Thm3.4}\eqref{Thm3.4b} and classical
Burkholder-Gundy-Davis inequality,  we obtain
\begin{eqnarray}\|a^{k*}\|_{p(\cdot)}&=&\|a^{k*}\chi_{A_k}\|_{p(\cdot)}\leq C\|a^{k*}\|_2\|\chi_{A_k}\|_{q(\cdot)}\nonumber\\
&\leq& C\|S(a^k)\chi_{A_k}\|_2\|\chi_{A_k}\|_{q(\cdot)}\nonumber\\
&\leq& C\|\lambda_\infty\chi_{A_k}\|_2\|\chi_{A_k}\|_{q(\cdot)}\nonumber\\
&\leq&\frac{C\|\chi_{A_k}\|_2\|\chi_{A_k}\|_{q(\cdot)}}{\|\chi_{A_k}\|_{p(\cdot)}}\leq C,\label{eq4.5}
\end{eqnarray}
where $\lambda$ is $f$'s optimal predictable
control in $\mathcal Q_{p(\cdot)}.$

It follows from \eqref{eq2.6} and \eqref{eq1.6} that $f^*\leq \sum_{k\in Z}\theta_k a^{k*}$ and
\begin{eqnarray*}
\|f^*\|^\eta_{p(\cdot)}&\leq& C\sum_{k\in Z} \|\theta_k a^{k*}\|^\eta_{p(\cdot)}\leq C\sum_{k\in Z} \theta_k^\eta\| a^{k*}\|^\eta_{p(\cdot)}\\
                       &\leq& C\sum_{k\in Z}\theta_k^\eta \leq C(\sum_{k\in Z}\theta_k^{p^+})^\frac{\eta}{p^+}.
\end{eqnarray*}
Combining this with \eqref{eq2.7}, we get $$\| f \|_{H^*_{p(\cdot)}}
\leq  C\| f \|_{\mathcal{Q}_{p(\cdot)}}.$$

Similarly, we get $$\|f\|_{H^S_{p(\cdot)}} \leq
C\|f\|_{\mathcal{D}_{p(\cdot)}}.$$
\end{proof}

\section{Some inequalities in spaces
with $0<p^-\leq p^+<\infty$}\label{sec5}

In the last section we extend some martingale inequalities in
the spaces with $0<p^-\leq p^+\leq \eta$ and the spaces with $1\leq
p^-\leq p^+<\infty$ to the spaces with $0<p^-\leq p^+<\infty.$

\begin{theorem}\label{Thm5.1}
Let $p\in \mathcal{P}.$ Then
\begin{equation}
\label{eq5.2} \| f \|_{H^*_{p(\cdot)}} \lesssim\| f \|_{\mathcal{Q}_{p(\cdot)}}, \forall f=(f_n).
\end{equation}
\end{theorem}

\begin{proof}[Proof of Theorem \ref{Thm5.1}] We apply an idea due to Chevalier \cite{Chevalier} (see also Weisz \cite{Weisz1}
) and split the proof into two steps.

$Step~1.$ To prove that if \eqref{eq5.2} holds for $p/2$, then it
also holds for $p.$ For this purpose, we suppose that
\begin{equation}\label{eq5.3}
\|f\|_{H^*_{p(\cdot)/2}} \leq  C\| f \|_{\mathcal{Q}_{p(\cdot)/2}},~~
\forall f\in \mathcal{Q}_{p(\cdot)/2}.
\end{equation}
Let $f\in \mathcal{Q}_{p(\cdot)}.$ We define $g_n=f^2_n-S^2_n(f).$ Then
$dg_n=2f_{n-1}df_n$ and $g=(g_n)$ is a martingale with
$$S_n(g)^2\leq 4f^{*2}_{n-1}S_n(f)^2\leq 4f^{*2}_{n-1}\lambda_{n-1}^2,$$
where $(\lambda_n)$ is a predictable control of $(S_n(f)).$ Using \eqref{eq1.5}, we have
$$\|g\|_{\mathcal{Q}_{p(\cdot)/2}} \leq 2\|f^*\lambda_\infty\|_{p(\cdot)/2}
\leq 2C_1\|f^*\|_{p(\cdot)}\|\lambda_\infty\|_{p(\cdot)}.$$
Thus
\begin{equation}\label{eq5.4}
\|g\|_{\mathcal{Q}_{p(\cdot)/2}}
\leq 2C_1\|f^*\|_{p(\cdot)}\|f\|_{\mathcal{Q}_{p(\cdot)}}.
\end{equation}

Since $f_n^2=g_n+S^2_n(f),$ we have $|f_n| \leq
|g_n|^{\frac{1}{2}}+S_n(f).$ Moreover,
\begin{equation}\label{eq5.5}
\|f\|_{H^*_{p(\cdot)}}\leq K(\|g^{*\frac{1}{2}}\|_{p(\cdot)}+\|S(f)\|_{p(\cdot)})
=K(\|g\|^{\frac{1}{2}}_{H^*_{p(\cdot)/2}}+\|f\|_{\mathcal{Q}_{p(\cdot)}}).
\end{equation}
It is clear that $f\in
\mathcal{Q}_{p(\cdot)}$ implies $f\in \mathcal{Q}_{p(\cdot)/2}.$ Combining \eqref{eq5.3} and
\eqref{eq5.4}, we get
$$\|g\|_{H^*_{p(\cdot)/2}}\leq C\|g\|_{\mathcal{Q}_{p(\cdot)/2}}\leq 2CC_1\|f^*\|_{p(\cdot)}\|f\|_{\mathcal{Q}_{p(\cdot)}}.$$
It follows from \eqref{eq5.5} that
$$\|f\|_{H^*_{p(\cdot)}}-K(\|f\|_{\mathcal{Q}_{p(\cdot)}}+
\sqrt{2CC_1}\|f\|^{\frac{1}{2}}_{H^*_{p(\cdot)}}\|f\|^{\frac{1}{2}}_{\mathcal{Q}_{p(\cdot)}})\leq0,$$
Let $z=\|f\|^{\frac{1}{2}}_{H^*_{p(\cdot)}}.$ Solving the quadratic inequality of $z$, we have
$$\|f\|_{H^*_{p(\cdot)}}\leq K^2(2CC_1+1)^2\|f\|_{\mathcal{Q}_{p(\cdot)}},$$
where $K^2(2CC_1+1)^2$ is a constant depending only on $p.$

$Step~2.$ Let $0<p^-\leq p^+<\infty.$ We take positive integer $m$
such that $2^{-m} p^+\leq \eta.$ It follows from Theorem
\ref{Thm4.2} that
$$\|f\|_{H^*_{2^{-m}p(\cdot)}}\leq C\|f\|_{\mathcal{Q}_{2^{-m}p(\cdot)}},~\forall f=(f_n).$$
Using Step 1, we have
$$\|f\|_{H^*_{2^{-m+1}p(\cdot)}}\leq C\|f\|_{\mathcal{Q}_{2^{-m+1}p(\cdot)}},~\forall f=(f_n).$$
Then by induction, we have
$$\|f\|_{H^*_{p(\cdot)}}\leq C\|f\|_{\mathcal{Q}_{p(\cdot)}},~\forall f=(f_n).$$
\end{proof}

\begin{theorem}\label{Thm5.2} Let $p\in \mathcal{P}.$ Then
\begin{equation}\label{eq5.6}
\|f\|_{H^s_{p(\cdot)}} \lesssim\| f \|_{\mathcal{D}_{p(\cdot)}},~~
\|f\|_{H^S_{p(\cdot)}} \lesssim\| f \|_{\mathcal{D}_{p(\cdot)}},~~~~
\forall f=(f_n).\end{equation}
\end{theorem}

\begin{proof}[Proof of Theorem \ref{Thm5.2}]  ~ $Step~1.$ To prove that if \eqref{eq5.6} holds for $2p,$ then it
also holds for $p.$ Let $f\in \mathcal{D}_{p(\cdot)}$ and
$\lambda=(\lambda_n)$ be $f$'s optimal predictable control. we
define $dg_n= \frac{df_n}{\sqrt{\lambda_{n-1}}},~\forall n\geq1.$
Then
$$g_n= \sum_{i=1}^n\frac{f_i-f_{i-1}}{\sqrt{\lambda_{i-1}}}
=\frac{f_n}{\sqrt{\lambda_{n-1}}}+\sum_{i=1}^{n-1}\frac{f_i}{\sqrt{\lambda_{i-1}\lambda_i}}(\sqrt{\lambda_i}-\sqrt{\lambda_{i-1}}),$$
and
\begin{equation}\label{eq5.7}
|g_n|\leq 2\sqrt{\lambda_{n-1}}, ~~ g^*\leq 2\sqrt{\lambda_\infty},
~~ Eg^{*2}\leq 4E\lambda_\infty<\infty.
\end{equation}
Thus $g=(g_n)$ is an
$L^2$-bounded martingale, which converges to
$g_\infty=\sum_{i=1}^\infty\frac{df_i}{\sqrt{\lambda_{i-1}}}$ a.e.
and in $L^2.$ Notice that $df_n= \sqrt{\lambda_{n-1}} dg_n,~ \forall
n\geq1, $ then
\begin{equation}
\label{eq8}s^2_n(f)\leq \lambda_{n-1} s^2_n(g),~~S^2_n(f)\leq \lambda_{n-1} S^2_n(g).
\end{equation}
For $s(f),$ by Lemma \ref{Lem1.4} and \eqref{eq1.5}, we get
$$\|s(f)\|_{p(\cdot)}\leq\|\lambda_\infty^{\frac{1}{2}}s(g)\|_{p(\cdot)}
\leq C\|\lambda_\infty^{\frac{1}{2}}\|_{2p(\cdot)}\|g\|_{H^s_{2p(\cdot)}}.$$
It follows from \eqref{eq5.7} that
\begin{equation}\label{eq5.9}\|s(f)\|_{p(\cdot)}\leq C\|\lambda_\infty^{\frac{1}{2}}\|_{2p(\cdot)}\|g\|_{\mathcal{D}_{2p(\cdot)}}
\leq C\|\lambda_\infty^{\frac{1}{2}}\|^2_{2p(\cdot)}=C\|\lambda_\infty\|_{p(\cdot)}.
\end{equation}
Thus $\|s(f)\|_{H^s_{p(\cdot)}} \leq
C\|f\|_{\mathcal{D}_{p(\cdot)}}.$

Similarly, we have $\|f\|_{H^S_{p(\cdot)}} \leq  C\| f
\|_{\mathcal{D}_{p(\cdot)}}.$

$Step~2.$ Let $0<p^-\leq p^+<\infty.$ We take positive integer $m$
such that $2^{m} p^-\geq 1.$ It follows from Theorem \cite[Theorem
4.4]{Liu1} that
\begin{equation*}
\|f\|_{H^s_{2^{m}p(\cdot)}} \leq  C\| f \|_{\mathcal{D}_{2^{m}p(\cdot)}},
\|f\|_{H^S_{2^{m}p(\cdot)}} \leq  C\| f \|_{\mathcal{D}_{2^{m}p(\cdot)}},
\forall f=(f_n).
\end{equation*}
Using Step 1, we have
\begin{equation*}
\|f\|_{H^s_{2^{m-1}p(\cdot)}} \leq  C\| f \|_{\mathcal{D}_{2^{m-1}p(\cdot)}},
\|f\|_{H^S_{2^{m-1}p(\cdot)}} \leq  C\| f \|_{\mathcal{D}_{2^{m-1}p(\cdot)}},
\forall f=(f_n).
\end{equation*}
Then by induction, we have
$$\|f\|_{H^s_{p(\cdot)}}\leq C\|f\|_{\mathcal{D}_{p(\cdot)}},~~ \|f\|_{H^S_{p(\cdot)}}\leq C\|f\|_{\mathcal{D}_{p(\cdot)}}, ~~~~ \forall f=(f_n).$$
\end{proof}

\begin{theorem}\label{Thm 5.3} Let $p\in \mathcal{P}.$ Then $\mathcal{Q}_{p(\cdot)}\thicksim\mathcal{D}_{p(\cdot)}.$
\end{theorem}

\begin{proof}[Proof of Theorem \ref{Thm 5.3}] Let $f=(f_n)\in \mathcal{D}_{p(\cdot)}$
and $\lambda=(\lambda_n)$ be $f$'s optimal predictable control. Then
$$S_n(f)\leq S_{n-1}(f)+|df_n|\leq S_{n-1}(f)+2\lambda_{n-1}.$$
It is clear that $(S_{n-1}(f)+2\lambda_{n-1})_{n\geq0}$ is a
predictable control of $(S_n(f))_{n\geq0}$ and
$f\in \mathcal{Q}_{p(\cdot)}.$ It follows from \eqref{eq5.6} that
$$\|f\|_{\mathcal{Q}_{p(\cdot)}}\leq \|S(f)+2\lambda_\infty\|_{p(\cdot)}
\leq K(\|f\|_{H^S_{p(\cdot)}}+2\|\lambda_\infty\|_{p(\cdot)})
\leq C\|f\|_{\mathcal{D}_{p(\cdot)}}.$$

On the other hand, let $f=(f_n)\in \mathcal{Q}_{p(\cdot)}$ and
$\lambda=(\lambda_n)$ be $f$'s optimal predictable control. Then
$$|f_n|\leq |f_{n-1}|+|df_n|\leq f^*_{n-1}+2\lambda_{n-1}$$ and
$f\in \mathcal{D}_{p(\cdot)}.$ Using \eqref{eq5.2}, we obtain
$$\|f\|_{\mathcal{D}_{p(\cdot)}}\leq \|f^*+2\lambda_\infty\|_{p(\cdot)}
\leq K(\|f\|_{H^*_{p(\cdot)}}+2\|\lambda_\infty\|_{p(\cdot)})
\leq C\|f\|_{\mathcal{Q}_{p(\cdot)}}.$$
\end{proof}

In \cite{Weisz1}, Weisz proved that if $(\Sigma_n)$ is regular, then all the
spaces $H^*_p,~ H^S_p,~ H^s_p, ~ \mathcal{D}_p$ and $\mathcal{Q}_p $
are equivalent when $0<p<\infty.$ Recently, Liu and Wang \cite{Liu1} proved its
variable exponent analogue for the case $1\leq p^-\leq p^+<\infty.$  In the rest of this chapter, we
extend the result to the case $0< p^-\leq p^+<\infty.$ Recall that a
martingale $f=(f_n)$ is previsible, if there is a real number $R>0$
such that
\begin{equation}\label{eq6.10} |df_n|^2\leq RE_{n-1}|df_n|^2, ~~ \forall n\geq 0. \end{equation}
Weisz \cite[Lemma 2.18]{Weisz1} showed that the assumption \eqref{eq6.10}
can be defined with the exponent $p$ instead of $2.$ In addition, Weisz \cite[Proposition 2.19]{Weisz1}
also showed that \eqref{eq6.10} holds for all martingale with the same
constant $R$ if and only if $(\Sigma_n)$ is regular.

\begin{theorem}\label{Thm5.4}  If $p\in \mathcal{P}$ and $(\Sigma_n)$ is regular,
then
$$H^*_{p(\cdot)}\sim H^S_{p(\cdot)}\sim
H^s_{p(\cdot)}\sim \mathcal{D}_{p(\cdot)}\sim\mathcal{Q}_{p(\cdot)}.$$
\end{theorem}

\begin{proof}[Proof of Theorem \ref{Thm5.4}] In view of Theorem \ref{Thm 5.3}, we have $\mathcal{D}_{p(\cdot)}\sim
\mathcal{Q}_{p(\cdot)}.$ Following from Theorems \ref{Thm5.1} and \ref{Thm5.2}, we obtain that
$\mathcal{Q}_{p(\cdot)}\hookrightarrow H^*_{p(\cdot)},
\mathcal{D}_{p(\cdot)}\hookrightarrow H^s_{p(\cdot)}.$
By regularity, it is easy to see that
$S_n(f)\leq R s_n(f)$ and $S(f)\leq R s(f),$ then
$H^s_{p(\cdot)}\hookrightarrow H^S_{p(\cdot)}.$
We still need to prove that $H^S_{p(\cdot)}\hookrightarrow
\mathcal{Q}_{p(\cdot)}$ and $H^*_{p(\cdot)}\hookrightarrow
\mathcal{D}_{p(\cdot)}.$ The proofs of $H^S_{p(\cdot)}\hookrightarrow
\mathcal{Q}_{p(\cdot)}$ and $H^*_{p(\cdot)}\hookrightarrow
\mathcal{D}_{p(\cdot)}$ are similar and we only prove $H^S_{p(\cdot)}\hookrightarrow
\mathcal{Q}_{p(\cdot)}.$

To prove $H^S_{p(\cdot)}\hookrightarrow
\mathcal{Q}_{p(\cdot)}.$ Let
$\|S(f)\|_{p(\cdot)}=1.$ By regularity, we have
$$S_n(f)\leq S_{n-1}(f)+|df_n|\leq S_{n-1}(f)+ RE_{n-1}|df_n| \leq S_{n-1}(f)+ RE_{n-1}S_n(f).$$
Then $(S_n(f))_{n\geq0}$ has a predictable control and
\begin{equation}\label{eq6.1}\|f\|_{\mathcal{Q}_{p(\cdot)}}\lesssim \|S(f)\|_{p(\cdot)}+
R\|\sup_{n\geq0}E_{n-1}S_n(f)\|_{p(\cdot)}.
\end{equation}
Let $1\leq p^-\leq p^+<\infty.$
Following from the proof of \cite[Theorem 4.5]{Liu1}, we have
\begin{equation}\label{eq6.1a}\|(\sup_{n\geq0}E_{n-1}S_n(f)) \|_{p(\cdot)}\lesssim 1.\end{equation}
To finish the proof of the theorem, we prove the following lemma
which might be useful in some other circumstances.

\begin{lemma}\label{lemma6.5}Let $1\leq p^-\leq p^+<\infty$
and $q$ be a real number with $0<q<1.$ Then $$\|(\sup_{n\geq0}E_{n-1}S_n(f))^q \|_{p(\cdot)}\lesssim 1.$$
\end{lemma}

\begin{proof}[Proof of Lemma \ref{lemma6.5}] Since
$$E_{n-1}S_n(f) = S_{n-1}(f) + E_{n-1}(S_n(f)-S_{n-1}(f)),$$
we have
$$(\sup_{n\geq0}E_{n-1}S_n(f))^q \leq S(f)^q
+ (\sum_{n\geq0}E_{n-1}(S_n(f)-S_{n-1}(f)))^q.$$
In view of convexity lemma for variable exponent
martingales (see \cite[Lemma 2.3]{Liu1}), we have
$$\|\sum_{n\geq0}E_{n-1}(S_n(f)-S_{n-1}(f))\|_{p(\cdot)}\lesssim \|S(f)\|_{p(\cdot)}=1.$$
It follows from Lemma \ref{Lem1.1} that
\begin{eqnarray*}
\rho_{p(\cdot)}((\sum_{n\geq0}E_{n-1}(S_n(f)-S_{n-1}(f)))^q)
&=&E(\sum_{n\geq0}E_{n-1}(S_n(f)-S_{n-1}(f)))^{qp}\\
&\leq& (E(\sum_{n\geq0}E_{n-1}(S_n(f)-S_{n-1}(f)))^p)^q\\
&=& \rho^q_{p(\cdot)}(\sum_{n\geq0}E_{n-1}(S_n(f)-S_{n-1}(f)))\lesssim 1.
\end{eqnarray*}
Using again Lemma \ref{Lem1.1}, we have
$$\|\sum_{n\geq0}E_{n-1}(S_n(f)-S_{n-1}(f)))^q\|_{p(\cdot)} \lesssim 1.$$
Thus
\begin{eqnarray*}
\|(\sup_{n\geq0}E_{n-1}S_n(f))^q \|_{p(\cdot)}&\lesssim&
(\|S(f)^q\|_{p(\cdot)}+ 1) \\
&=& (\|S(f)\|^q_{qp(\cdot)}+ 1)
\lesssim 1. \end{eqnarray*}
\end{proof}

Let $0<p^-<1.$ We replace $q$ and $p$ by
$p^-$ and $p/p^-$ in Lemma \ref{lemma6.5} respectively. It follows that
$$\|\sup_{n\geq0}E_{n-1}S_n(f)\|^{p-}_{p(\cdot)}
=\|(\sup_{n\geq0}E_{n-1}S_n(f))^{p-}\|_{p(\cdot)/p^-} \lesssim 1.$$
Combining this with \eqref{eq6.1a} and using \eqref{eq6.1}, we have
$$\|f\|_{\mathcal{Q}_{p(\cdot)}}\lesssim 1,$$
where $0< p^-\leq p^+<\infty.$ Thus $H^S_{p(\cdot)}\hookrightarrow
\mathcal{Q}_{p(\cdot)}.$ This finishes the proof of the theorem.
\end{proof}

%
%

\end{document}